\documentclass{amsart}

\usepackage[latin1]{inputenc}       
\usepackage{amsfonts}
\usepackage{euscript}
\usepackage{amssymb}
\usepackage{mathrsfs}

\usepackage{graphicx}
\usepackage{epsfig}

\newcommand{\sub}{\underline}
\newcommand{\real}{\mathbb{R}}
\newcommand{\nat}{\mathbb{N}}
\newcommand{\hiper}{\mathbb{H}}
\newcommand{\esf}{\mathbb{S}}
\newcommand{\var}{\mathbb{M}}
\newcommand{\bol}{\mathbb{B}}
\newcommand{\lo}{\mathbb{L}}

\newcommand{\arcsec}{\: \mathrm{arcsec} \:}
\newcommand{\interior}{\: \mathrm{int} \:}
\newcommand{\graph}{\: \mathrm{graph} \:}
\newcommand{\diverg}{\: \mathrm{div} \:}
\newcommand{\osc}{\: \mathrm{osc} \:}
\newcommand{\dist}{\: \mathrm{dist_{\hiper^2}} \:}
\newcommand{\disto}{\: \mathrm{dist_{\hiper^2\times\{0\}}} \:}
\newcommand{\distn}{\: \mathrm{dist_{\hiper^n}} \:}

\newcommand{\g}{\mathcal{G}}
\newcommand{\ca}{\mathcal{C}}
\newcommand{\de}{\mathcal{D}}
\newcommand{\ho}{\mathcal{H}}

\newcommand{\ce}{\mathscr{C}}
\newcommand{\h}{\mathscr{H}}
\newcommand{\dom}{\mathscr{D}}

\theoremstyle{plain}
\newtheorem{thm}{Theorem}[section]
\newtheorem{prop}[thm]{Proposition}

\theoremstyle{definition}
\newtheorem{dfn}[thm]{Definition}

\theoremstyle{remark}
\newtheorem{rem}[thm]{Remark}

\numberwithin{equation}{section}
\numberwithin{thm}{section}

\begin{document}


\title[COMPACT MINIMAL VERTICAL GRAPHS WITH NON-CONNECTED BOUNDARY]{COMPACT MINIMAL VERTICAL GRAPHS WITH NON-CONNECTED BOUNDARY IN $\mathbb{H}^n\times\mathbb{R}$}
\author{ALINE MAURICIO BARBOSA}

\address{Departamento de Matemática\\
Instituto de Ciências Exatas\\
Universidade Federal Rural do Rio de Janeiro\\
BR 465, km 7 - 23890-000 - Seropédica, RJ\\
Brazil}

\email{alinanet@ufrrj.br}

\thanks{The author would like to thank CAPES Agency, for partial financial support, and Maria Fernanda Elbert, for her suggestions, critical reading and encouragement during the preparation of this paper.}

\begin{abstract}
We study the existence and uniqueness problem of compact minimal vertical graphs in $\mathbb{H}^n\times\mathbb{R}$, $n\geq 2$, over bounded domains in the slice $\mathbb{H}^n\times\{0\}$, with non-connected boundary having a finite number of $C^0$ hypersufaces homeomorphic to the sphere $\mathbb{S}^{n-1}$, with prescribed bounded continuous boundary data, under hypotheses relating those data and the geometry of the boundary. We show the nonexistence of compact minimal vertical graphs in $\mathbb{H}^n\times\mathbb{R}$ having the boundary in two slices and the height greater than or equal to $\pi/(2n-2)$. 

\vspace*{0.3 cm}


\end{abstract}

\maketitle


\section{Introduction}

\indent\par In this paper, we consider the product space $\hiper^n\times\real$. We study the existence and the uniqueness of minimal graphs over bounded domains of the slice $\hiper^n\times\{0\}$ with non-connected boundary having a finite number of $C^0$ hypersurfaces which are homeomorphic to $\esf^{n-1}$ and satisfy determined interior or exterior sphere conditions or convexity condition, with prescribed bounded continuous boundary data. (For the precise definitions of \textit{interior/ exterior sphere condition} or \textit{convexity condition}, see Definition \ref{def:ceie}.)

We denote by $\distn(p_1,p_2)$ the hyperbolic distance between two points $p_1$, $p_2\in\hiper^n$ and as usual define the hyperbolic distance between two sets $A_1$, $A_2\subset\hiper^n$ by
$$\distn(A_1,A_2)=\inf\,\{\distn(p_1,p_2); p_1\in A_1, \, p_2\in A_2\}.$$

For each $r>0$, let us set 
$$\g_r(\rho)=\int_r^{\rho} \frac{\sinh^{n-1} r}{\sqrt{\sinh^{2n-2}\xi-{\sinh}^{2n-2} r}} \; d\xi, \;\; \rho\in[r,+\infty).$$

We shall see in Subsection \ref{subsec:exgrmin} that the function above define the generator curve of an $n$-dimensional half-catenoid in $\hiper^n\times\real$.

The first theorem of this paper guarantees the existence of compact minimal hypersurfaces with boundary in two slices of $\hiper^n\times \real$, given as vertical graphs over bounded domains with non-connected boundary in $\hiper^n\times \{0\}$, under hypotheses which relate the distance between those slices and the geometric nature of the boundary of these domains. More precisely:

\begin{thm}
Let $\Gamma_1$, ..., $\Gamma_k$, $\Gamma$ be $C^0$ hypersurfaces in the slice $\hiper^n\times\{0\}$, which are homeomorphic to $\esf^{n-1}$. Assume that each hypersurface $\Gamma_i$, $i=1$, ..., $k$, is contained in the interior of the region of $\hiper^n\times\{0\}$ bounded by $\Gamma$ and that the closed regions bounded by the hypersurfaces $\Gamma_i$ are pairwise disjoint. Assume also that each $\Gamma_i$,  $i=1$, ..., $k$, satisfies the interior sphere condition of some radius $0<R_1<\infty$ and that $\Gamma$ satisfies the exterior sphere condition of some radius $0<R_2<\infty$, if $n\geq 2$, or of radius $R_2=\infty$, if $n=2$.

\noindent Denote by:
\begin{eqnarray}
\delta & = & \distn \left(\bigcup_{i=1}^k \Gamma_i, \Gamma\right),\nonumber\\
 & & \nonumber\\
\Omega & = & \mbox{the domain in }\hiper^n\times\{0\} \mbox{ bounded by }\Gamma_i,\; i=1,\; ...,\; k, \nonumber\\
 & & \mbox{ and by }\Gamma,\nonumber\\
 & & \nonumber\\
\alpha_1 & = & \g_{R_1}(R_1+\delta),\label{eq:alfar1}\\
 & & \nonumber\\
\alpha_2 & = & \left\{
\begin{array}{l}
{\displaystyle \g_{R_2}(R_2+\delta),  \mbox{ if }  0<R_2<\infty \mbox{ and } n\geq 2,}\\ 
 \\
\arcsec (e^{\delta}),  \mbox{ if }  R_2=\infty \mbox{ and } n=2.
\end{array}
\right.\label{eq:alfar2}
\end{eqnarray}
If $h$ is a real number such that
\begin{equation}
0\leq h\leq \min_{j=1,2}\alpha_j
\label{eq:alfamin}
\end{equation}
then there is a unique compact minimal hypersurface in $\hiper^n\times\real$, given as the vertical graph of a function $u\in C^2(\Omega)\cap C^0(\bar{\Omega})$ such that $\left. u \right|_{\Gamma}=0$ and $\left. u \right|_{\Gamma_i}=h$, $i=1$, ..., $k$.
\label{teo:ac}
\end{thm}

Theorem \ref{teo:ac} is a version for $\hiper^n\times\real$ of Theorem 2.1 of \cite{er} (p. 606-607) in $\real^3$. In the mentioned theorem of \cite{er}, N. Espirito-Santo and J. Ripoll have shown the existence of a compact minimal graph with boundary having a finite number of Jordan curves in two parallel planes of $\real^3$, under geometric conditions given for those curves and for the distance $h$ between these parallel planes. It is important to mention that the problem of existence of minimal surfaces (and more generally, of constant mean curvature surfaces) in $\real^3$ with boundary given by curves contained in parallel planes has also been studied in other papers, see for instance \cite{fr}, \cite{afr} and \cite{af}. These papers differ from \cite{er}, because they consider graphs over domains of the Euclidean sphere $\esf^2\subset\real^3$ (radial graphs), while in \cite{er}, the authors have worked with graphs over planar domains.

The second theorem of this paper guarantees the existence of compact minimal hypersurfaces in $\hiper^n\times \real$, given as vertical graphs over bounded domains with non-connected boundary in $\hiper^n\times\{0\}$, with bounded continuous boundary data which are not necessarily locally constant, under hypotheses which relate these data and the geometric nature of the boundary of these domains, see Theorem \ref{teo:d}.

\begin{thm}
Let $\Gamma_1$, ..., $\Gamma_k$, $\Gamma$ be $C^0$ hypersurfaces in $\hiper^n\times\{0\}$, which are homeomorphic to $\esf^{n-1}$. Assume that each hypersurface $\Gamma_i$, $i=1$, ..., $k$, is contained in the interior of the region of $\hiper^n\times\{0\}$ bounded by $\Gamma$ and that the closed regions bounded by the hypersurfaces $\Gamma_i$ are pairwise disjoint. Assume also that each $\Gamma_i$,  $i=1$, ..., $k$, satisfies the interior sphere condition of some radius $0<R<\infty$ and that the hypersurface $\Gamma$ is convex.

\noindent Denote by:
\begin{eqnarray}
\delta & = & \distn \left(\bigcup_{i=1}^k \Gamma_i, \Gamma\right),\nonumber\\
 & & \nonumber\\
\Omega & = & \mbox{the domain in }\hiper^n\times\{0\} \mbox{ bounded by }\Gamma_i,\; i=1,\; ...,\; k, \nonumber\\
 & & \mbox{ and by }\Gamma,\nonumber\\
 & & \nonumber\\
\alpha & = & \g_R(R+\delta).\label{eq:alfarn}
\end{eqnarray}
Suppose that $f:\Gamma\rightarrow\real$ is a continuous function such that
$$\osc f \leq \alpha,$$
where $\osc f:=\max f-\min f$ is the oscillation of $f$ in $\Gamma$.

\noindent If $h$ is a real number such that
\begin{equation}
\max f\leq h \leq \min f+\alpha,
\label{eq:oscn}
\end{equation}
then there is a unique compact minimal hypersurface in $\hiper^n\times\real$, given as the vertical graph of a function $u\in C^2(\Omega)\cap C^0(\bar{\Omega})$ such that $\left. u \right|_{\Gamma}=f$ and $\left. u \right|_{\Gamma_i}=h$, $i=1$, ..., $k$.
\label{teo:d}
\end{thm}

As in \cite{er}, the main ingredient for the proof of Theorems \ref{teo:ac} and \ref{teo:d} of this paper is Perron process, described in Subsection \ref{subsec:perron}. This process requires the construction of barriers at each boundary point of the considered domain, to guarantee that the solution $u\in C^2(\Omega)$ provided by the process for the minimal equation in $\hiper^n\times\real$ extends continuously up to the boundary and satisfies the boundary data condition. (The precise definition of barriers will be established in Subsection \ref{subsec:perron}.) What allows the construction of barriers in the proof of Theorem \ref{teo:ac} is the existence of rotational minimal hypersurfaces (slices and $n$-dimensional catenoids) in $\hiper^n\times\real$ presented, for instance, in the papers \cite{nr}, \cite{st:screw} and \cite{nsst} for $n=2$ and in \cite{bs} for $n\geq 2$, and the existence of a minimal surface foliated by horizontal horocycles, presented in \cite{d}. The construction of barriers for the proof of Theorem \ref{teo:d} is possible due to the existence of Scherk type minimal hypersurfaces in $\hiper^n\times\real$, presented, for instance, in \cite{nr} and \cite{st:screw} for $n=2$ and in \cite{st:hnxr} for $n\geq 3$. We emphasize that Theorem \ref{teo:ac} is not consequence of Theorem \ref{teo:d}, since the exterior sphere condition for a compact without boundary hypersurface in $\hiper^n$ does not imply its convexity. (This fact will be shown in Section \ref{sec:cond}.)

Also in the paper \cite{er}, N. Espirito-Santo and J. Ripoll have obtained a non-existence result for connected compact minimal graphs with boundary in parallel planes of $\real^3$, defined over a anullar domain in one of these planes, under hypotheses which relate the distance between these planes to the diameter of the outside boundary curve of the domain. (See Proposition 3.1 in \cite{er}, p. 615, to know the precise result.)

We obtain in our paper the following non-existence result in $\hiper^n\times\real$:

\begin{prop}
Let $\Pi_1:=\hiper^n\times\{0\}$ and $\Pi_2:=\hiper^n\times\{h\}$ be two slices of $\hiper^n\times \real$, where $h>0$. Let $\alpha\subset \Pi_1$ and $\beta\subset \Pi_2$ be hypersurfaces homeomorphic to $\esf^{n-1}$, such that the orthogonal projection $\beta^*$ of $\beta$ on the slice $\Pi_1$ is in the  domain enclosed by $\alpha$. If $h\geq\pi/(2n-2)$, then it does not exist a connected compact minimal vertical graph $M$ over the domain $\Omega\subset\Pi_1$ bounded by $\alpha\cup\beta^*$, with boundary $\partial M=\alpha\cup\beta$.
\label{prop:inexn}
\end{prop}

The main ingredient of the proof of the above result is the fact that the height of an $n$-dimensional half-catenoid in $\hiper^n\times \real$ never exceeds $\pi/(2n-2)$.

This paper is part of Doctoral Thesis of the author \cite{ba}, at Universidade Federal do Rio de Janeiro.

\section{Preliminaries}
\label{sec:prel}

\subsection{Vertical graphs in $\hiper^n\times\real$}
\label{subsec:grvert}


\indent\par In the product manifold $\hiper^n\times\real$, we consider the ball model for the $n$-dimensional hyperbolic space $\hiper^n$. Denoting by $x_1,\ldots,x_n$ the coordinates in $\hiper^n$ and by $t$ the coordinate in $\real$, then we consider that $\hiper^n\times\real$ is the set
$$\{(x_1,\ldots, x_n,t)\in \real^{n+1}; \; x_1^2+ \cdots+ x_n^2<1\},$$
endowed with the product metric
$$d\sigma^2=\frac{dx_1^2+\cdots+dx_n^2}{F}+dt^2,$$
where
$$F=\left(\frac{1-(x_1^2+\cdots+x_n^2)}{2}\right)^2.$$

The knowledge about $n$-dimensional hyperbolic geometry is fundamental for the comprehension of this paper. For this, we recommend, for instance, Chapters 2 and 3 of \cite{st:cours}.

\begin{dfn}
\textbf{(Vertical graph) } Let $\Omega \subset \hiper^n$ be a domain and let $u:\Omega\rightarrow\real$ be a $C^2$ function in $\Omega$. The \textit{vertical graph} of $u$ is the subset of $\hiper^n\times \real$ given by
$$\{(x_1,\ldots,x_n,u(x_1,\ldots,x_n)); (x_1,\ldots,x_n)\in \Omega\}.$$
\label{def:grafvertn}
\end{dfn}

\begin{prop}
{\rm \textbf{(Mean curvature equation in $\hiper^n\times \real$)} } Let $\Omega\subset\hiper^n$ be a domain. The vertical graph of a function $u\in C^2(\Omega)$ has constant mean curvature $H$ if and only if $u$ satisfies the following equation:
$$\diverg \left(\frac{\nabla u}{\tau_u}\right)+\frac{(n-2)\,p\cdot\nabla u}{\tau_u \sqrt{F}}=\frac{nH}{F},$$
where
$$F=\left(\frac{1-\left\|p\right\|^2}{2}\right)^2, \; p=(x_1,\ldots,x_n)\in \Omega,$$
$$\tau_u=\sqrt{1+F\left\|\nabla u\right\|^2},$$
$\cdot$, $\left\|\;\right\|$, $\nabla$ and $\diverg$ are, respectively, the scalar product, the norm, the gradient and the divergence in the Euclidean metric of $\real^n$, and the mean curvature $H$ is obtained with respect to the unit normal vector field $N$ to the graph of $u$ with positive $(n+1)$-th component.
\label{prop:cmchn}
\end{prop}

The proof of Proposition \ref{prop:cmchn} follows classical procedures, similar to the particular case $n=2$ in \cite{nr}, p. 264-265. The reader interested in the calculation for the general case $n\geq 2$ can see, for instance, the Appendix A in the thesis \cite{ba}, p. 77-88.

\indent\par Given $H\in\real$, we define the operator $Q_H$ by
$$Q_H(u):=\diverg \left(\frac{\nabla u}{\tau_u}\right)+\frac{(n-2)\,p\cdot\nabla u}{\tau_u \sqrt{F}}-\frac{nH}{F}, \; \; u\in C^2(\Omega),$$
where $\Omega$ is a domain in $\hiper^n\times\{0\}$. It follows from Proposition \ref{prop:cmchn} that the vertical graph of a function $u\in C^2(\Omega)$ has constant mean curvature $H$ if and only if $u$ satisfies the equation $Q_H \,u=0$.

For the particular case $H=0$, we have:
\begin{equation}
Q_0(u):=\diverg \left(\frac{\nabla u}{\tau_u}\right)+\frac{(n-2)\,p\cdot\nabla u}{\tau_u \sqrt{F}}, \; \; u\in C^2(\Omega).
\label{eq:q0n}
\end{equation}
Then, the vertical graph of a function $u\in C^2(\Omega)$ is minimal if and only if $u$ satisfies the equation $Q_0 \,u=0$.

Observe that the above operator is elliptic, in the divergence form. We can write
$$Q_0(u)=\diverg \left(\mathbf{A}(p,u,\nabla u)\right)+B(p,u,\nabla u), \; \; u\in C^2(\Omega),$$
where
$$\mathbf{A}(p,u,\nabla u)= \frac{\nabla u}{\tau_u}$$ and $$B(p,u,\nabla u)=\frac{(n-2)\,p\cdot\nabla u}{\tau_u \sqrt{F}}$$
are continuously differentiable functions with respect to the variable $\nabla u$ and independent of the variable $u$ (They only depend on $p$ and $\nabla u$). Consequently, we can use Theorem 10.7 in \cite{gt} to state the Maximum Principle for the minimal hypersuface equation $Q_0 \,u=0$.

\begin{thm}
{\rm \textbf{(Maximum Principle) }} Let $\Omega \subset \hiper^n$ be a bounded domain, where $\partial \Omega$ denotes the boundary of $\Omega$. Let $f_1, \, f_2:\partial \Omega \rightarrow \real$ be continuous functions satisfying $f_1\leq f_2$. Let $u_i:\bar{\Omega}\rightarrow \real$ be a continuous extension of $f_i$ satisfying the minimal hypersurface equation $Q_0 \,u=0$ on $\Omega$, $i=1, \, 2$. Then we have $u_1\leq u_2$ on $\Omega$.
\label{teo:pmaxn}
\end{thm}

\begin{proof}
See, for instance, Theorem 10.7 of \cite{gt}, p. 268-271.
\end{proof}

As a consequence of Theorem \ref{teo:pmaxn}, setting $f_1=f_2$, there exists at most one continuous extension of $f_1$ on $\bar{\Omega}$ satisfying the minimal hypersurface equation $Q_0 \,u=0$ on $\Omega$.

\subsection{Examples of minimal vertical graphs in $\hiper^n\times\real$}
\label{subsec:exgrmin}


\indent\par For the sake of clearness, we present in this subsection some useful examples of minimal vertical graphs in $\hiper^n\times\real$ found in the literature.\\

\noindent \textbf{A) Rotational minimal vertical graphs in $\hiper^n\times \real$}

\indent\par The papers \cite{nr}, \cite{st:screw}, \cite{nsst} and \cite{bs} describe the construction of rotational minimal hypersurfaces generated around a vertical geodesic axis in $\hiper^n\times\real$ ($n$-dimensional catenoids and slices). The first three papers deal with the particular case $n=2$ and the last one generalizes it for $n\geq 2$.

More precisely, given $r>0$, the function
\begin{equation}
\g_r(\rho)=\int_r^{\rho} \frac{\sinh^{n-1} r}{\sqrt{\sinh^{2n-2}\xi-{\sinh}^{2n-2} r}} \; d\xi, \;\; \rho\in[r,+\infty)
\label{eq:gr}
\end{equation}
define the generator curve of an $n$-dimensional half-catenoid in $\hiper^n\times\real$, where $\rho$ represents the (hyperbolic) horizontal distance with respect to the rotational axis, see Proposition 3.2 in \cite{bs}.

Fixed a point $c_0\in \hiper^n$, we denote by $\rho_{c_0}(q)$ the hyperbolic distance from a point $q\in \hiper^n$ to $c_0$, that is,
$$\rho_{c_0}(q):=\distn(q,c_0).$$

Observe that the function $\rho \mapsto \g_r(\rho)$ is continuous and increasing on $[r,+\infty)$ and that an $n$-dimensional half-catenoid with the generator curve defined by $\g_r$ and the rotation axis $c_0 \times \real$ in $\hiper^n\times\real$ is the vertical graph of the function
\begin{equation}
\g_{c_0,r}(q)=\int_r^{\rho_{c_0}(q)} \frac{\sinh^{n-1} r}{\sqrt{\sinh^{2n-2}\xi-{\sinh}^{2n-2} r}} \; d\xi,
\label{eq:gcr}
\end{equation}
over the set $\{q\in \hiper^n\times\{0\}; \;\; \rho_{c_0}(q)\geq r\}$, that is, over the exterior domain to the hyperbolic $(n-1)$-sphere with center $c_0$ and radius $r$ in $\hiper^n\times\{0\}$.

Also by Proposition 3.2 in \cite{bs}, the half-catenoid defined by $\g_{c_0,r}$ has finite vertical height
$$h^+(r)=\int_r^{+\infty} \frac{\sinh^{n-1} r}{\sqrt{\sinh^{2n-2}\xi-{\sinh}^{2n-2} r}} \; d\xi.$$

In addition, the function $r\mapsto h^+(r)$ increases from $0$ to ${\displaystyle \frac{\pi}{2(n-1)}}$, when $r$ increases from $0$ to $+\infty$. This means that the height of an $n$-dimensional half-catenoid  in $\hiper^n\times\real$ never exceeds ${\displaystyle \frac{\pi}{2(n-1)}}$. Moreover, given $r_1,\,r_2>0$, with $r_1\neq r_2$, the generator curves $\g_{r_1}$ and $\g_{r_2}$ intersect at one unique point.

If $r_0>r$, then $\g_r(r_0)$ represents the height of the $n$-dimensional catenoid part which is vertical graph of $\g_{c_0,r}$ over the closed domain of $\hiper^n\times\{0\}$ bounded by hyperbolic $(n-1)$-spheres with radii $r$ and $r_0$ and common center $c_0$.

A geometric description with more details about the $n$-dimensional catenoids in $\hiper^n\times\real$ can be seen in Subsections 3.2 and 3.3 of \cite{bs}.\\

\noindent \textbf{B) Scherk type minimal vertical graphs in $\hiper^n\times\real$}

\indent\par In \cite{nr}, B. Nelli and H. Rosenberg have proved the existence of minimal graphs in $\hiper^2\times\real$, defined over geodesic triangles in $\hiper^2$, with infinite boundary data on one of the sides of the triangle and zero value boundary data on the other two sides. More precisely, they have proved the following:

\begin{thm}
{\rm \textbf{(Theorem 2 in \cite{nr})} } Let $\Delta$ be a geodesic triangle in $\hiper^2$ with sides $A$, $B$ and $C$.
Then there exists a function $u$, solution of the minimal equation $Q_0=0$, defined in $\Delta\setminus A$, which satisfies
$$\left.u\right|_{\interior (B\cup C)}=0, \mbox{ } \lim_{q\rightarrow \interior A}u(q)=+\infty.$$
Moreover, $|\nabla u(q)|\rightarrow +\infty$ when $q$ approaches the side  $A$.
\label{teo:nrscherk}
\end{thm}

\begin{proof}
See, for instance, Theorem 2 in \cite{nr}, p. 269-273. (A alternative proof is given in \cite{st:asymp}, Example 4.3, p. 326.)
\end{proof}

The vertical graph of $u$ determined by the above theorem (as well as any vertical graph which is isometric to it) will be called a \textit{Scherk type minimal graph in $\hiper^2\times\real$}.

\indent\par In \cite{st:hnxr}, R. Sa Earp and E. Toubiana have proved the existence of Scherk type minimal graphs in $\hiper^n\times\real$, for $n\geq 3$, that is, of minimal vertical graphs, defined over certain bounded domains of $\hiper^n$, which assume infinite boundary data on certain parts of this boundary. (Theorems 5.1, 5.3 and 5.4 in \cite{st:hnxr}.)

We shall present the Scherk type minimal graph presented in Theorem 5.1 in \cite{st:hnxr}, after the two following definitions.

\begin{dfn}
Let $\Sigma\subset\hiper^n$ be a $C^0$ orientable hypersurface in $\hiper^n$.

\begin{enumerate}

\item We say that $\Sigma$ is \textit{convex in the hyperbolic sense} if, at each point $p\in\Sigma$, $\Sigma$ is contained in one of the closed halfspaces determined by some geodesic hyperplane of $\hiper^n$ passing through $p$.
	
\item We say that $\Sigma$ is \textit{strictly convex in the hyperbolic sense} if, at each point $p\in\Sigma$, $\Sigma$ is contained in one of the closed halfspaces determined by some geodesic hyperplane $\Pi_p$ of $\hiper^n$ passing through $p$ and, moreover, $\Sigma\cap\Pi_p=\{p\}$.

\end{enumerate}

\label{def:hipersupconvex}
\end{dfn}

\begin{dfn}
\textbf{(Definition 5.1 in \cite{st:hnxr} - Special rotational domain) } Let $\gamma, L\subset \hiper^n$ be, $n\geq 3$, two complete geodesics such that $L$ is orthogonal to $\gamma$ at some point $B\in \gamma\cap L$.

Using the half-space model for $\hiper^n$, we can assume, without loss of generality, that $\gamma$ is the vertical geodesic such that $\partial_{\infty} \gamma=\{0,\infty\}$.

We call $P\subset\hiper^n$ the geodesic two-plane containing $L$ and $\gamma$. We choose points $A_0\in(0,B)\subset\gamma$ and $A_1\in L\setminus\gamma$ and we denote by $\alpha\subset P$ the Euclidean segment joining $A_0$ and $A_1$.

Therefore the hypersurface $\Sigma$, generated by rotating $\alpha$ with
respect to $\gamma$, has the following properties:
\begin{enumerate}
	\item $\interior(\Sigma)$ is smooth except at point $A_0$;
	\item $\Sigma$ is strictly convex in hyperbolic meaning and convex in Euclidean meaning;
	\item $\interior(\Sigma)\setminus\{A_0\}$ is transversal to the Killing field generated by the translations along $\gamma$.
\end{enumerate}

Consequently $\Sigma$ lies in the mean convex side of the domain of $\hiper^n$ whose boundary is the hyperbolic cylinder with axis $\gamma$ and passing through $A_1$.

Let us call $\Pi\subset\hiper^n$ the geodesic hyperplane orthogonal to $\gamma$ and passing through $B$. Observe that the boundary of $\Sigma$ is a $(n-2)$-dimensional geodesic sphere of $\Pi$ centered at $B$.

We denote by $U_{\Sigma}\subset\Pi$ the open geodesic ball centered at $B$ whose boundary is the boundary of $\Sigma$. We call $\de_{\Sigma}\subset\hiper^n$ the closed domain whose boundary is $U_{\Sigma}\cup\Sigma$. Observe that $\partial\de_{\Sigma}$ is strictly convex at any point of $\Sigma$ and convex at any point of $U_{\Sigma}$. Such a domain will be called a \textit{special rotational domain}.
\label{def:dre}
\end{dfn}

\begin{thm}
{\rm \textbf{(Theorem 5.1 in \cite{st:hnxr})} } Let $\de_{\Sigma}\subset\hiper^n$ be a special rotational domain. There is a unique solution $v_{\infty}$ of the minimal equation $Q_0=0$ in $\interior(\de_{\Sigma})$, which extends continuously to $\interior(\Sigma)$ taking prescribed zero boundary value data and taking boundary value $+\infty$ for any approach to $U_{\Sigma}$.
More precisely, the following Dirichlet problem {\rm\textbf{(P$_{\infty}$)}} admits a unique solution $v_{\infty}$:
$$
\mbox{{\rm\textbf{(P$_{\infty}$)}}}\left\{\begin{array}{l}
Q_0(u)=0 \mbox{ in } \interior(\de_{\Sigma}),\\
u=0 \mbox{ on } \interior(\Sigma),\\
u=+\infty \mbox{ on } U_{\Sigma},\\
u\in C^2(\interior(\de_{\Sigma}))\cap C^0(\de_{\Sigma}\setminus\overline{U_{\Sigma}}).	
\end{array}\right.
$$
\label{teo:scherkrot}
\end{thm}

\begin{proof}
See Theorem 5.1 (and also Proposition 5.1) in \cite{st:hnxr}.
\end{proof}

The vertical graph of $v_{\infty}$ determined by the above theorem (as well as any vertical graph which is isometric to it) will be called a \textit{rotational Scherk hypersurface in $\hiper^n\times\real$}.
\\

\noindent \textbf{C) Minimal vertical graphs foliated by horocycles in $\hiper^2\times\real$}

\indent\par In the paper \cite{d}, B. Daniel has exhibited a minimal surface in $\hiper^2\times \real$ such that each horizontal curve is a horocycle in the slice where it lies, in such a way that the vertical projections of these horocycles on the slice $\hiper^2\times \{0\}$ have the same asymptotic point.

Before we present precisely the above result, we observe that the mentioned paper has considered the Minkowsky model for $\hiper^2$, that is, the hyperboloid
$$\hiper^2_{\var}=\{(x_0,x_1,x_2)\in\lo^3; \; -x_0^2+x_1^2+x_2^2=-1, \; x_0>0\},$$
endowed with the quadratic form
$$g_{\var}=-(dx_0)^2+(dx_1)^2+(dx_2)^2.$$
($\lo^3$ denotes the three-dimensional Lorentz space, that is, the space $\real^3$ endowed with the form $g_{\var}$.)

\begin{prop}
{\rm \textbf{(Proposition 4.17 in \cite{d})} } The map
\begin{equation}
\chi(u,v)=\left(\frac{v^2+1}{2 \cos u}+\frac{\cos u}{2},\;\frac{v}{\cos u},\;\frac{v^2-1}{2 \cos u}+\frac{\cos u}{2},\;u\right),
\label{eq:c0m}
\end{equation}
defined for ${\displaystyle (u,v)\in \left(-\frac{\pi}{2},\frac{\pi}{2}\right)\times\real}$, is a conformal minimal embedding in $\hiper^2_{\var}\times\real$, such that the curves $u=u_0$, ${\displaystyle u_0\in\left(-\frac{\pi}{2},\frac{\pi}{2}\right)}$, are horocycles in $\hiper^2_{\var}\times\{u_0\}$ such that its vertical projections over the slice $\hiper^2_{\var}\times\{0\}$ have the same asymptotic point. We will denote this surface by $\ca_0$.

Morover, the surface $\ca_0$ is the unique one (up to isometries of $\hiper^2_{\var}\times\real$) having this property.
\label{prop:horomin}
\end{prop}

\begin{proof}
See Proposition 4.17 in \cite{d}, p. 6277-6278.
\end{proof}

Considering the Poincaré disk model for $\hiper^2$, that is,
$$\hiper^2_{\bol}=\{(x,y)\in \real^2; \; x^2+y^2<1\},$$
endowed with the metric
$$g_{\bol}=\left(\frac{2}{1-(x^2+y^2)}\right)^2(dx^2+dy^2),$$
and using the fact that the map
$$
\begin{array}{l}
\Pi:\hiper^2_{\var}\rightarrow\hiper^2_{\bol}\\
{\displaystyle \Pi\;(x_0,x_1,x_2)=\left(\frac{x_1}{1+x_0},\frac{x_2}{1+x_0}\right), \; (x_0,x_1,x_2)\in \hiper^2_{\var},}
\end{array}
$$
is a isometry between the models $(\hiper^2_{\var},g_{\var})$ and $(\hiper^2_{\bol},g_{\bol})$ of $\hiper^2$ (see, for instance, \cite{st:cours}, p. 205-206),
we get
\begin{equation}
X(u,v)=\left(\frac{2v}{v^2+(1+\cos u)^2},\;\frac{v^2-1+\cos^2 u}{v^2+(1+\cos u)^2},\;u\right),
\label{eq:c0b}
\end{equation}
defined for ${\displaystyle (u,v)\in \left(-\frac{\pi}{2},\frac{\pi}{2}\right)\times\real}$, is a parameterization for the surface $\ca_0$ in $\hiper^2\times\real$, where $\hiper^2=(\hiper^2_{\bol},g_{\bol})$, corresponding to the expression (\ref{eq:c0m}) obtained in $\hiper^2_{\var}\times\real$.

To see clearly that the curves $u=u_0$ are horocycles in $\hiper^2\times\{u_0\}$ such that its vertical projections over the slice $\hiper^2\times\{0\}$ have the same asymptotic point, we can rewrite the expression (\ref{eq:c0b}), denoting by

\begin{eqnarray}
x & = & {\displaystyle \frac{2v}{v^2+(1+\cos u)^2}} \label{eq:xhoro}
\end{eqnarray}
and by
\begin{eqnarray}
y & = & {\displaystyle \frac{v^2-1+\cos^2 u}{v^2+(1+\cos u)^2}}, \label{eq:yhoro}
\end{eqnarray}
eliminating the parameter $v$ from the expressions above and obtaining, after several calculations and simplifications, that the equations (\ref{eq:xhoro}) and (\ref{eq:yhoro}) result in
\begin{eqnarray}
{\displaystyle x^2+\left(y-\frac{\cos u}{1+\cos u}\right)^2} & = & \frac{1}{(1+\cos u)^2}\cdot \label{eq:ceuclid}
\end{eqnarray}

We fixe ${\displaystyle u_0\in\left(-\frac{\pi}{2},\frac{\pi}{2}\right)}$. We conclude from (\ref{eq:xhoro}), (\ref{eq:yhoro}) and (\ref{eq:ceuclid}) that the horizontal curve $X(u_0,\cdot)$ of the surface $\ca_0$ 
is contained in the Euclidean circle with center ${\displaystyle \left(0,\frac{\cos u_0}{1+\cos u_0},u_0\right)}$ and radius ${\displaystyle \frac{1}{1+\cos u_0}}$ on the horizontal plane $u=u_0$, that is, it is contained in the Euclidean circle parameterized by
$$\beta_{u_0}(\theta)=\left(\frac{\cos \theta}{1+\cos u_0},\frac{\cos u_0+\sin \theta}{1+\cos u_0}, u_0\right), \; \theta\in\left[\frac{\pi}{2},\frac{5\pi}{2}\right].$$

Note that ${\displaystyle \beta_{u_0}\left(\frac{\pi}{2}\right)=\beta_{u_0}\left(\frac{5\pi}{2}\right)=(0,1,u_0)\in\partial_{\infty}(\hiper^2\times\{u_0\})=\esf^1\times\{u_0\}}$.

We denote by $\tilde{\beta}_{u_0}$ the restriction of $\beta_{u_0}$ on the open interval ${\displaystyle \left(\frac{\pi}{2},\frac{5\pi}{2}\right)}$. $\tilde{\beta}_{u_0}$ parameterizes the Euclidean circle with center ${\displaystyle \left(0,\frac{\cos u_0}{1+\cos u_0},u_0\right)}$ and radius ${\displaystyle \frac{1}{1+\cos u_0}}$ minus the point $(0,1,u_0)$ on the horizontal plane $u=u_0$.

Note that the point $(0,1,u_0)$ does not belong to the curve $X(u_0,\cdot)$, because, otherwise, we would have $v=0$ and $\cos u_0=-1$, but this do not occurs for ${\displaystyle u_0\in \left(-\frac{\pi}{2},\frac{\pi}{2}\right)}$. Thus the curve $X(u_0,\cdot)$ is contained in $\tilde{\beta}_{u_0}$. On the other hand, the curve $\tilde{\beta}_{u_0}$ is contained in the horizontal curve $X(u_0,\cdot)$. In fact, given ${\displaystyle \theta\in\left(\frac{\pi}{2},\frac{5\pi}{2}\right)}$, if we take ${\displaystyle v_{\theta}=(1+\cos u_0)\frac{\cos \theta}{1-\sin \theta}}$, we get $X(u_0,v_{\theta})=\tilde{\beta}_{u_0}(\theta)$. Therefore $\tilde{\beta}_{u_0}$ is a reparameterization for the horizontal curve $X(u_0,\cdot)$.

Now it is easy to conclude that the horizontal curve $\tilde{\beta}_{u_0}$ parameterizes, in fact, a horocycle in $\hiper^2\times\{u_0\}$ with asymptotic point $(0,1,u_0)\in \partial_{\infty}(\hiper^2\times\{u_0\})$, for each ${\displaystyle u_0\in \left(-\frac{\pi}{2},\frac{\pi}{2}\right)}$.

Thus,
\begin{equation}
Y(u,\theta)=\left(\frac{\cos \theta}{1+\cos u},\frac{\cos u+\sin \theta}{1+\cos u},u\right), \; (u,\theta)\in\left(-\frac{\pi}{2},\frac{\pi}{2}\right)\times\left(\frac{\pi}{2},\frac{5\pi}{2}\right)
\label{eq:c0b2}
\end{equation}
is a reparameterization for (\ref{eq:c0b}) and hence, it is a new parameterization for the surface $\ca_0$.

Therefore the surface $\ca_0$ is foliated by the horocycles $u=u_0$, ${\displaystyle u_0\in\left(-\frac{\pi}{2},\frac{\pi}{2}\right)}$, such that its vertical projections over the slice $\hiper^2\times\{0\}$ are horocycles having the same asymptotic point $(0,1,0)$.

Note that $\ca_0$ has the height $\pi$. (Observe that the horocycles $u=u_0$ converge to ${\displaystyle \esf^1\times\left\{\frac{\pi}{2}\right\}}$ when ${\displaystyle u_0\rightarrow\frac{\pi}{2}}$ and to ${\displaystyle \esf^1\times\left\{-\frac{\pi}{2}\right\}}$ when ${\displaystyle u_0\rightarrow-\frac{\pi}{2}}$.)

Moreover, $\ca_0$ is symmetric with respect to the slice $\hiper^2\times\{0\}$. In fact, $Y(-u,\theta)=R(Y(u,\theta))$ for all ${\displaystyle u\in\left[0,\frac{\pi}{2}\right)}$, where $R$ is the reflection with respect to the slice $\hiper^2\times \{0\}$.

Observe also that $\ca_0$ is invariant by a one-parameter family of horizontal parabolic isometries which let globally fixed each horocycle of $\hiper^2\times\{0\}$ tangent to the asymptotic point $(0,1,0)\in\partial_{\infty}(\hiper^2\times\{0\})=\esf^1\times\{0\}$ and, therefore, they let globally fixed each horocycle $u=u_0$, ${\displaystyle u_0\in\left(-\frac{\pi}{2},\frac{\pi}{2}\right)}$.


We denote by $\ho_0$ the horocycle in $\hiper^2\times\{0\}$ parameterized by $Y(0,\cdot)$ and by $\dom$ the closed connected region of $(\hiper^2\times\{0\})\setminus\ho_0$ which has $\partial_{\infty}(\hiper^2\times\{0\})=\esf^1\times\{0\}$ as asymptotic boundary.

Considering each horizontal horocycle $Y(u,\cdot)$ of the surface $\ca_0$ as a Euclidean circle minus the point $(0,1,u)$, observe that its Euclidean radius ${\displaystyle \frac{1}{1+\cos u}}$ increases from ${\displaystyle \frac{1}{2}}$ to $1$ when $u$ varies from $0$ to ${\displaystyle \frac{\pi}{2}}$. From this follows that the vertical projections of the horocycles $Y(u,\cdot)$ over $\hiper^2\times\{0\}$ cover $\dom$ exactly once when $u$ varies in the interval ${\displaystyle \left[0,\frac{\pi}{2}\right)}$.

Hence the upper half $\ca_0^+$ of the surface $\ca_0$, that is,
$$\ca_0^+=\left\{Y(u,\theta); \; (u,\theta)\in\left[0,\frac{\pi}{2}\right)\times\left(\frac{\pi}{2},\frac{5\pi}{2}\right)\right\},$$
where $Y(u,\theta)$ is given in (\ref{eq:c0b2}), is a vertical graph of some real function over $\dom$. We denote this function by $\Upsilon$.


\subsection{Perron process for the minimal hypersurface equation in $\hiper^n\times\real$}
\label{subsec:perron}


\indent\par Perron process will be the principal ingredient for the solution of the proposed problems in this paper. Notations, definitions and results of this subsection are based in the works \cite{st:hnxr}, \cite{st:h3} (p. 682-685) and \cite{st:asymp} (p. 321-325). We include them here to make this paper self contained.

\indent\par We consider the following Dirichlet problem:
$$\mbox{\textbf{(P) }}\left\{
\begin{array}{cc}
Q_0 \,u=0 \mbox{ in } \Omega, & u\in C^2(\Omega)\cap C^0(\bar{\Omega}),\\
  & \\
\left. u \right|_{\partial \Omega}=f, & 
\end{array}
\right.$$
where $\Omega$ is any bounded domain in $\hiper^n\times\{0\}\equiv \hiper^n$,  $Q_0 \,u$ is defined in (\ref{eq:q0n}) and $f:\partial \Omega\rightarrow\real$ is any bounded continuous function, defined in $\partial \Omega$.

Let $v:\bar{\Omega}\rightarrow\real$ be a continuous function.
Let $U\subset \Omega$ be a closed ball in $\hiper^n\times\{0\}$. Then ${\displaystyle \left. v\right|_{\partial U}}$ has a unique minimal extension $\tilde{v}_U$ on $U$, continuous up to the boundary $\partial U$.
In fact, since $U$ is a bounded domain having the boundary $\partial U$ of class $C^2$ with positive mean curvature when the normal vector field to $\partial U$ points to $U$ ($\partial U$ is a hyperbolic $(n-1)$-sphere), it follows from Theorem 1.5 of J. Spruck in \cite{s} (p. 787-788) that the Dirichlet problem
$$\left\{
\begin{array}{cc}
Q_0 \,u=0 \mbox{ in } U, & u\in C^2(U)\cap C^0(\bar{U}),\\
  & \\
\left. u \right|_{\partial U}=\varphi, & 
\end{array}
\right.$$
has a unique solution for arbitrary continuous boundary data $\varphi$ (and, particularly, for $\varphi=\left. v \right|_{\partial U}$).

Therefore the function $M_U(v)$, given by
$$M_U(v)(q)=\left\{
\begin{array}{rl}
v(q), & \mbox{ if } q\in \bar{\Omega}\setminus U,\\
  & \\
\tilde{v}_U(q), & \mbox{ if } q\in U,
\end{array}
\right.$$
is well defined and it is continuous in $\bar{\Omega}$.

\begin{dfn}
\textbf{(\cite{st:asymp}; \cite{st:hnxr}) } We say that $v\in C^0(\bar{\Omega})$ is a \textit{subsolution} (respectively \textit{supersolution}) of the problem \textbf{(P)} if
\begin{enumerate}
	\item For any closed ball $U\subset \Omega$, we have $v\leq M_U(v)$ (respectively $v\geq M_U(v)$),
	\item ${\displaystyle \left. v\right|_{\partial \Omega}\leq f}$ (respectively ${\displaystyle \left. v\right|_{\partial \Omega}\geq f}$).
\end{enumerate}
\label{def:subsupersol}
\end{dfn}

We present below some classical properties about subsolutions and supersolutions (see, for instance, \cite{ch}, p. 307-309):

\begin{enumerate}

	\item If $v\in C^2(\Omega)$, the condition 1 in the above definition is equivalent to $Q_0(v)\geq 0$ for subsolution or $Q_0(v)\leq 0$ for supersolution.
	
	\item If $v, \, w\in C^0(\Omega)$ are functions such that $v\geq w$, then $M_U(v)\geq M_U(w)$ for all closed ball $U\subset \Omega$. (This is a consequence of the Maximum Principle.)
	
	\item If $v, \, w\in C^0(\Omega)$ are two subsolutions (respectively, supersolution) of \textbf{(P)} then the function $u\in C^0(\Omega)$, defined by $u(q)=\max(v(q),w(q))$, $q\in \Omega$ (respectively, $u(q)=\min(v(q),w(q))$, $q\in \Omega$), again is a subsolution (respectively, supersolution) of \textbf{(P)}.
	
	\item If $v$ is a subsolution (respectively, supersolution) of \textbf{(P)} and $U\subset \Omega$ is a closed ball then $M_U(v)$ again is a subsolution (respectively, supersolution) of \textbf{(P)}.
	
	\item Let $v, \, w:\bar{\Omega}\rightarrow \real$ be two continuous functions such that $M_U(v)\geq v$ and $M_U(w)\leq w$ for any closed ball $U\subset \Omega$. Suppose that ${\displaystyle \left.v\right|_{\partial\Omega}\leq \left.w\right|_{\partial\Omega}}$. Then $v\leq w$ on $\Omega$. Roughly speaking, a supersolution is greater than a subsolution.
	
\end{enumerate}

\begin{dfn}
\textbf{(Definition 4.2 in \cite{st:hnxr} - Barriers) } We consider the Dirichlet problem \textbf{(P)}. Let $p\in \partial\Omega$ be a boundary point.
\begin{enumerate}
	\item Suppose that, for any $M>0$  and for any $k\in \nat$, there are an open neighborhood $V_k$ of $p$ in $\hiper^n$ and a function $\omega_k^+$ (respectively, $\omega_k^-$) in $C^2(V_k\cap\Omega)\cap C^0(\overline{V_k\cap \Omega})$ such that
	
\begin{enumerate}
	\item ${\displaystyle \left.\omega_k^+\right|_{\partial\Omega\cap V_k}\geq f}$ and ${\displaystyle \left.\omega_k^+\right|_{\partial V_k\cap \Omega}\geq M}$
	
	(respectively, ${\displaystyle \left.\omega_k^-\right|_{\partial\Omega\cap V_k}\leq f}$ and ${\displaystyle \left.\omega_k^-\right|_{\partial V_k\cap \Omega}\leq -M)}$;
	
	\item $Q_0 (\omega_k^+) \leq 0$ (respectively, $Q_0 (\omega_k^-) \geq 0$) in $V_k\cap \Omega$;
	
	\item ${\displaystyle \lim_{k\rightarrow +\infty} \omega_k^+(p)=f(p)}$ (respectively, ${\displaystyle \lim_{k\rightarrow +\infty} \omega_k^-(p)=f(p)}$).
\end{enumerate}

\item Suppose that there exists a supersolution $\phi$ (respectively, a subsolution $\eta$) in $C^2(\Omega)\cap C^0(\overline{\Omega})$ such that $\phi(p)=f(p)$ (respectively, $\eta(p)=f(p)$).

\end{enumerate}
In both cases, 1 or 2, we say that $p$ admits an \textit{upper barrier} $\{\omega_k^+\}_{k\in \nat}$ or $\phi$ (respectively, \textit{lower barrier} $\{\omega_k^-\}_{k\in \nat}$ or $\eta$) for the problem \textbf{(P)}.

\label{def:barreiras}
\end{dfn}

\begin{dfn}
\textbf{(Definition 4.3 in \cite{st:hnxr} - $C^0$ convex domains) } Let $\Omega\subset\hiper^n$ be a $C^0$ domain.

\begin{enumerate}

\item We say that $\Omega$ is \textit{convex} at $p\in \partial\Omega$ if there exists a neighborhood of $p$ in $\bar{\Omega}$ totally contained in one of the closed halfspaces determined by some geodesic hyperplane of $\hiper^n$ passing through $p$.
	
\item We say that $\Omega$ is \textit{strictly convex} at $p\in \partial\Omega$ if there exists a neighborhood $U_p\subset \bar{\Omega}$ of $p$ in $\bar{\Omega}$ which is totally contained in one of the closed halfspaces determined by some geodesic hyperplane $\Pi$ of $\hiper^n$ passing through $p$ such that $U_p\cap \Pi=\{p\}$.

\end{enumerate}

\label{def:convexn}
\end{dfn}

\begin{rem}
Considering the problem \textbf{(P)}, it is possible to obtain barriers at any point in $\partial \Omega$ where the domain $\Omega\subset\hiper^n$ is convex, for any bounded continuous boundary data $f$.

For the particular case $n=2$, R. Sa Earp and E. Toubiana have made this construction in \cite{st:asymp}, using sequences of Scherk type minimal graphs in $\hiper^2\times\real$, defined over geodesic triangles in $\hiper^2$. Specifically, for the construction of an upper barrier (in the sense of the Definition \ref{def:barreiras}-1) for the problem \textbf{(P)} at a point $p_0\in \partial \Omega$ where $\Omega$ is $C^0$ convex, the authors have exhibited a sequence of isosceles geodesic triangles $\Delta_k$, which are small enough in $\hiper^2$, where each $\Delta_k$ contains $p_0$ on its geodesic axis of symmetry, and they have defined Scherk type minimal graphs $\omega_k^+$ over $V_k:=\interior \: \Delta_k$, satisfying the following conditions, for each $k\in \nat$ and for any $M>0$:

(1) ${\displaystyle \left.\omega_k^+\right|_{\partial\Omega\cap V_k}\geq f}$ and ${\displaystyle \left.\omega_k^+\right|_{\partial V_k\cap \Omega}\geq M};$
	
(2) $Q_0 (\omega_k^+) = 0$ in $V_k\cap \Omega$;
	
(3)	${\displaystyle \omega_k^+(p_0)=f(p_0)+1/k}$.

Analogously, one obtain a lower barrier at $p_0$. For the reader which is interested in the details of the construction of these barriers, see Example 4.1 of  \cite{st:asymp}, p. 322-323.

For $n\geq 3$, the same authors have exhibited in \cite{st:hnxr} barriers at any point of $\partial \Omega$ where the domain $\Omega\subset\hiper^n$ is convex, for arbitrary bounded continuous boundary data $f$, using sequences of Scherk type minimal graphs in $\hiper^n\times\real$, defined over special rotational domains in  $\hiper^n$. The procedures adopted in this construction were similar to the case $n=2$. For the details about this construction, see Theorem 5.2 in \cite{st:hnxr}.
\label{obs:barreiras}
\end{rem}

\begin{thm}
{\rm\textbf{(Theorem 4.5 in \cite{st:hnxr} - Perron process) }} Let $\Omega\subset\hiper^n$ be a bounded domain and $f:\partial \Omega\rightarrow\real$ a bounded continuous function. Let $\phi$ be a bounded supersolution of the Dirichlet problem {\rm\textbf{(P)}}, for instance, the constant function $\phi\equiv \sup f$. Set
$$S_\phi=\{s\in C^0(\bar{\Omega}); \mbox{ s is subsolution of {\rm \textbf{(P)}}, with } s\leq \phi\}.$$
(Observe that $S_\phi\neq \emptyset$, since the constant function $\eta\equiv \inf f$ belongs to $S_\phi$.)
For each $q\in\bar{\Omega}$, we define
$$u(q):=\sup_{s\in S_\phi} s(q).$$
We have the following:

\begin{enumerate}
	\item The function $u$ is of class $C^2$ on $\Omega$ and satisfies the minimal equation $Q_0 \, u=0$.
	
	\item Let $p\in \partial\Omega$ be a boundary point and suppose that $p$ admits a barrier. Then the solution $u$ is continuous at $p$ and satisfies $u(p)=f(p)$.
\end{enumerate}
 
\label{teo:perron}
\end{thm}

\begin{proof}
See, for instance, Theorem 4.5 in \cite{st:hnxr}.
\end{proof}


\section{Interior/ exterior sphere condition and convexity in $\hiper^n$}
\label{sec:cond}
\vspace*{0,5 cm}

\begin{dfn}
Let $\Gamma$ be a $C^0$ hypersurface in $\hiper^n$, which is homeomorphic to $\esf^{n-1}$, and let $a$, $b>0$ be positive real numbers.
\begin{enumerate}
\item We say that $\Gamma$ satisfies the \textit{interior sphere condition of radius $a$} if, at each point $p\in\Gamma$, there exists a hyperbolic $(n-1)$-sphere of radius $a$ passing through $p$, contained in the closure of the bounded connected region of $\hiper^n\setminus\Gamma$.
	
\item We say that $\Gamma$ satisfies the \textit{exterior sphere condition of radius $b$} if, at each point $p\in\Gamma$, there exists a hyperbolic $(n-1)$-sphere of radius $b$ passing through $p$, such that its mean convex side is contained in the closure of the unbounded connected region of $\hiper^n\setminus\Gamma$.
	
\item We say that $\Gamma$ satisfies the \textit{exterior horosphere condition} if, at each point $p\in\Gamma$, there exists a horosphere passing through $p$, such that its mean convex side is contained in the closure of the unbounded connected region of $\hiper^n\setminus\Gamma$.
	
\item We say that $\Gamma$ is \textit{convex (in the hyperbolic sense)} if, at each point $p\in\Gamma$, there exists a geodesic hyperplane passing through $p$, such that $\Gamma$ is contained in one of the closed halfspaces of $\hiper^n$ determined by this hyperplane.
\end{enumerate}
\label{def:ceie}
\end{dfn}

Note that a hyperbolic $(n-1)$-sphere $S\subset \hiper^n$ of radius $r$ satisfies automatically the interior sphere condition of radius $a=r$.

Recall that a horosphere in $\hiper^n$ is the limit hypersurface of a sequence of hyperbolic $(n-1)$-spheres $S_{\rho_n}(q_n)$ having center $q_n$ and radius $\rho_n$, such that $\rho_n\rightarrow\infty$ and $q_n$ approaches the asymptotic point $p'\in\partial_{\infty}\hiper^n$ of this horosphere when $n\rightarrow\infty$. (See, for instance, \cite{v}, p. 37-38 or \cite{bj}, p. 101-102.) Due to this fact, the horospheres are considered as hyperbolic $(n-1)$-spheres of center at the infinity (namely, at its asymptotic point $p'\in\partial_{\infty}\hiper^n$) and infinite radius. Therefore, the \textit{exterior horosphere condition} above defined might have been recalled as \textit{exterior sphere condition of infinite radius}.

\begin{rem}
In Definition \ref{def:ceie}, if we consider the particular case $n=2$, $\Gamma$ is reduced to a $C^0$ Jordan curve contained in $\hiper^2$ and the expressions \textit{sphere} (or \textit{$(n-1)$-sphere}), \textit{horosphere}, \textit{geodesic hyperplane} and \textit{halfspace} are naturally replaced by \textit{circle}, \textit{horocycle}, \textit{geodesic line} and \textit{halfplane}, respectively.
\label{obs:ccie}
\end{rem}

\begin{prop}
Let $\Gamma$ be a $C^0$ hypersurface in $\hiper^n$, which is homeomorphic to $\esf^{n-1}$. If $\Gamma$ is convex then it satisfies the exterior sphere condition of infinite radius.
\label{prop:chec}
\end{prop}

\begin{proof}
If $\Gamma$ is convex then, given $p\in\Gamma$, there exists a geodesic hyperplane $\Pi$ passing through $p$, such that $\Gamma$ is contained in one of the closed halfspaces of $\hiper^n$ determined by this hyperplane.

We denote by $U_1$ the closed halfspace of $\hiper^n$ determined by $\Pi$  in which $\Gamma$ is contained and by $U_2$ the other halfspace.

Let $\gamma$ be the geodesic line of $\hiper^n$ orthogonal to $\Pi$ at $p$.

We denote by $\gamma_i:=U_i\cap\gamma$, $i=1,2$.

Let $p'\in \partial_{\infty}\hiper^n$ be the asymptotic point of $\gamma_2$. We denote by $\h_p$ the horosphere which is orthogonal to $\gamma$ at $p$ and which has asymptotic point $p'$. By construction, $\h_p$ is tangent to $\Pi$ at $p$.

Orienting $\h_p$ such that the normal vector field to this horosphere points to the connected component of $\hiper^n\setminus\h_p$ whose asymptotic boundary is $p'$, we get from the Maximum Principle that $\h_p$ must be contained in $U_2$.

Since $\Gamma\subset U_1$, it follows that the mean convex side of $\h_p$ is contained in the closure of the unbounded connected region of $\hiper^n\setminus\Gamma$.

Thus, $\Gamma$ satisfies the exterior horosphere condition, that is, the exterior sphere condition of infinite radius.
\end{proof}

\begin{rem}
A $C^0$ hypersurface $\Gamma\subset\hiper^n$ which is homeomorphic to $\esf^{n-1}$ may satisfy the exterior sphere condition of some radius $0<\rho<\infty$ and be non-convex. An example of this fact for the case $n=2$ is illustrated in Figure \ref{fig:feijaocce}.
	
In Figure \ref{fig:feijaocce}, we note that the Jordan curve $\Gamma$ satisfies the exterior sphere condition of some radius $0<\rho<\infty$. However, it is not convex in the hyperbolic sense. In fact, in the given example, the origin $(0,0)$ of $\hiper^2$ belongs to $\Gamma$ and we observe that all geodesic lines passing through $(0,0)$ divides $\Gamma$ in two parts at least.

It is not difficult to generalize the above example for $n\geq 2$.
\label{obs:cceconvex}
\end{rem}

\begin{figure}[htb]
\begin{center}
\includegraphics[scale=.4]{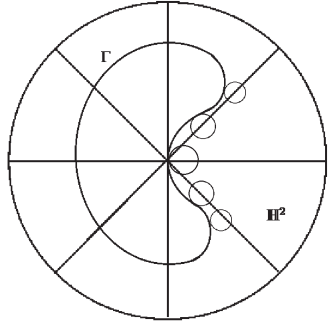}
\caption{An example of a Jordan curve in $\hiper^2$ which satisfies the exterior sphere condition of some radius $0<\rho<\infty$ and that is not convex.}
\label{fig:feijaocce}
\end{center}
\end{figure}

\begin{figure}[htb]
\begin{center}
\includegraphics[scale=.4]{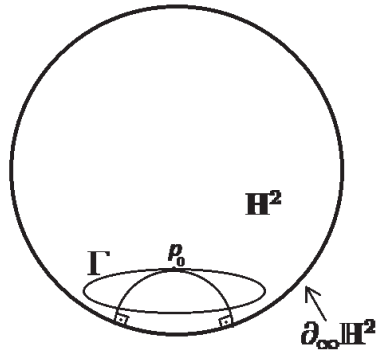}
\caption{An example of a Jordan curve in $\hiper^2$ which satisfies the exterior horosphere condition and that is not convex.}
\label{fig:elipseeuclid}
\end{center}
\end{figure}

\begin{rem}
The reciprocal of Proposition \ref{prop:chec} is false. For instance, for the case $n=2$, we take an Euclidean ellipse $\Gamma\subset\hiper^2$, near enough to $\partial_{\infty}\hiper^2$ in the Euclidean sense, as in Figure \ref{fig:elipseeuclid}.
	
It is easy to see that $\Gamma$ satisfies the exterior horosphere condition. However, the unique geodesic line $\alpha$ tangent to $\Gamma$ at $p_0$ divides $\Gamma$ in three parts, as we see in Figure \ref{fig:elipseeuclid}. (Of course the geodesic lines which are transversal to $\Gamma$ at $p_0$ will do the same.)

We may generalize this example for $n\geq 2$, taking as $\Gamma\subset\hiper^n$ a $(n-1)$-dimensional ellipsoid, near enough to  $\partial_{\infty}\hiper^n$ in the Euclidean sense. 
\label{obs:checonvex}
\end{rem}

\begin{rem}
We recall that, in $\real^n$, a $C^0$ hypersurface $\Gamma\subset\real^n$, which is homeomorphic to $\esf^{n-1}$, is said to be \textit{convex in the Euclidean sense} if, at each point $p\in\Gamma$, there exists an affine hyperplane passing through $p$, such that $\Gamma$ is contained in one of the closed halfspaces of $\real^n$ determined by this hyperplane. We recall also that a $C^0$ hypersurface $\Gamma\subset\real^n$, which is homeomorphic to $\esf^{n-1}$, satisfies the \textit{Euclidean exterior sphere condition of radius $b$} if, given any point $p\in\Gamma$, there exists a Euclidean $(n-1)$-sphere of radius $b$ passing through $p$, such that its mean convex side is contained in the closure of the unbounded connected region of $\real^n\setminus\Gamma$. We recall that the affine hyperplanes are the geodesic hyperplanes of $\real^n$ and also, they may be seen as the Euclidean $(n-1)$-spheres of infinite radius. From this follows that the hypersurface $\Gamma$ is convex in the Euclidean sense if and only if it satisfies the Euclidean exterior sphere condition of infinite radius.

However, in $\hiper^n$ there is not the equivalence between convexity and exterior sphere condition of infinite radius. In fact, at first, we recall that the horospheres (and not the geodesic hyperplanes) are the $(n-1)$-spheres of the infinite radius in $\hiper^n$. Now, according to Proposition \ref{prop:chec}, all $C^0$ hypersurface in $\hiper^n$ which is homeomorphic to $\esf^{n-1}$ and convex (in the sense hyperbolic) satisfies the (hyperbolic) exterior sphere condition of infinite radius. But, as we have seen in Remark \ref{obs:checonvex}, a $C^0$ hypersurface in $\hiper^n$ may satisfy the (hyperbolic)  exterior sphere condition of infinite radius and not be convex (in the hyperbolic sense). This means that, in $\hiper^n$, the exterior sphere condition of infinite radius is weaker than the hyperbolic convexity condition.
\label{obs:ceeri}
\end{rem}


\section{Proof of Theorem \ref{teo:ac}}
\label{sec:teoac}

\indent\par We want to solve the following problem:
$$\mbox{\textbf{(P$_0$) }}\left\{
\begin{array}{l}
Q_0 \, u=0 \mbox{ in } \Omega, \;\; u\in C^2(\Omega)\cap C^0(\bar{\Omega}),\\
\left. u \right|_{\Gamma}=0, \;\; \left. u \right|_{\Gamma_i}=h, \;\; i=1,\; ...,\; k.  
\end{array}
\right.$$

We note that the function $\phi\equiv h \mbox{ in } \bar{\Omega}$ is a supersolution for the problem \textbf{(P$_0$)}.

We define the set
$$S_0=\{s\in C^0(\bar{\Omega}); \; s \mbox{ is a subsolution of \textbf{(P$_0$)}, with } s\leq h\}.$$

The above set is not empty, because the function $s_0\equiv 0 \mbox{ in }\bar{\Omega}$ belongs to $S_0$.

We define the function
$$u(q):=\sup_{s\in S_0} s(q), \;\; q\in \bar{\Omega}.$$

It follows from Perron process that $u\in C^2(\Omega)$ and it satisfies $Q_0 \, u=0$. Moreover, by Perron process, we need to show that each point in $\partial\Omega$ admits barriers for the problem \textbf{(P$_0$)}. The barriers will be obtained in the sense of Definition \ref{def:barreiras}-2, that is, we shall show the existence of subsolutions $v_p$ and supersolutions $w_p$ of \textbf{(P$_0$)} such that $$v_p(p)=w_p(p)=h,$$
if ${\displaystyle p\in \bigcup_{i=1}^k \Gamma_i}$, and
$$v_p(p)=w_p(p)=0,$$ if $p\in \Gamma$.\\

Given $i=1,...,n$, let us consider a point $p\in\Gamma_i$.

We know that the function $w_p\equiv h$ in $\bar{\Omega}$ is a supersolution of \textbf{(P$_0$)} such that $w_p(p)=h$.
	
For the construction of a subsolution $v_p$ of \textbf{(P$_0$)} such that $v_p(p)=h$, we may consider a hyperbolic $(n-1)$-sphere $S_p$ with radius $R_1$ passing through $(p,h)$ and contained in the region of the slice $\hiper^n\times \{h\}$ bounded by the hypersurface $\Gamma_i+h \, e_{n+1}$, where $e_{n+1}=(0,...,0,1)$. We denote by $c_p$ the hyperbolic center of $S_p$ and by ${c_p}^*$ the vertical projection of $c_p$ on the slice $\hiper^n\times \{0\}$.
	
It follows from the hypotheses (\ref{eq:alfamin}) and (\ref{eq:alfar1}) that
$$\g_{R_1}(R_1)=0\leq h \leq \g_{R_1}(R_1+\delta),$$
where the function $\g_{R_1}$ is defined as in Formula (\ref{eq:gr}). Since $\g_{R_1}$ is continuous and increasing, we guarantee the existence of a unique $R_h \in [R_1,R_1+\delta]$ such that
$$h=\g_{R_1}(R_h),$$
and consequently, the construction of the piece of $n$-dimensional catenoid in $\hiper^n\times \real$ which is vertical graph of the function
$$-\g_{{c_p}^*, R_1}+ h, $$
over the set $\{q\in \hiper^n\times \{0\}; \; R_1\leq \rho_{{c_p}^*}(q)\leq R_h\}$. (Recall Formula (\ref{eq:gcr}).)\\

We note that this piece of catenoid has as boundary two hyperbolic $(n-1)$-spheres: $S_{1,p}\subset \hiper^n\times \{h\}$ and $S_{2,p}\subset \hiper^n\times \{0\}$, where $S_{1,p}=S_p$ and $S_{2,p}$ has center ${c_p}^*$ and radius $R_h$, see Figure \ref{fig:teoa-af1b-catenoide-def-ingles}.

\begin{figure}[htb]
\begin{center}
\includegraphics[scale=.53]{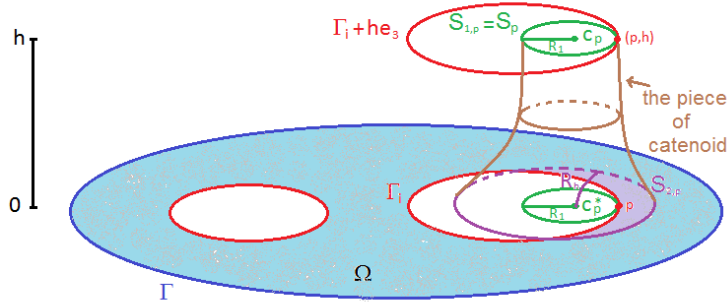}
\caption{The piece of catenoid (graph of $-\g_{{c_p}^*, R_1}+ h$) with boundary $S_{1,p}\cup S_{2,p}$ $(n=2)$.}
\label{fig:teoa-af1b-catenoide-def-ingles}
\end{center}
\end{figure}

Since $R_h\leq R_1+\delta$, $S_{2,p}$ is in the region of $\hiper^n\times \{0\}$ bounded by the hypersurface $\Gamma$, see again Figure \ref{fig:teoa-af1b-catenoide-def-ingles}.

The function $v_p$, defined in $\bar{\Omega}$ by
$$v_p(q)=\left\{
\begin{array}{rl}
-\g_{{c_p}^*, R_1}(q)+h, &  \mbox{ if } q\in\bar{\Omega} \mbox{ and it is in the region bounded by } \Gamma_i\\
 & \mbox{ and }S_{2,p},\\
0, & \mbox{ if } q\in\bar{\Omega} \mbox{ and it is not in the region bounded}\\
 & \mbox{ by } \Gamma_i \mbox{ and } S_{2,p},
\end{array}
\right.$$
is a subsolution for the problem \textbf{(P$_0$)}, see again Figure \ref{fig:teoa-af1b-catenoide-def-ingles}. In fact, by construction, $v_p\in C^0(\bar{\Omega})$ and $v_p$ is the maximum of two subsolutions of \textbf{(P$_0$)}, namely, $v_p=\max \{-\g_{{c_p}^*, R_1}+h, 0\}$ in $\bar{\Omega}$. Moreover, $v_p(p)=h$.\\

Now let us consider a point $p\in\Gamma$.

We know that the function $v_p\equiv 0$ in $\bar{\Omega}$ is a subsolution of \textbf{(P$_0$)} such that $v_p(p)=0$.

For the construction of a supersolution $w_p$ of \textbf{(P$_0$)} such that $w_p(p)=0$, we consider two cases:\\

\noindent \sub{Case 1}: $0<R_2<\infty$ and $n\geq 2$.

We may consider a hyperbolic $(n-1)$-sphere $Z_p$ with radius $R_2$, passing through $p$, such that the ball bounded by $Z_p$ is contained in the closure of the exterior of $\Gamma$ in $\hiper^n\times \{0\}$. We denote by $d_p$ the hyperbolic center of $Z_p$.

It follows from the hypotheses (\ref{eq:alfamin}) and (\ref{eq:alfar2}) that
$$\g_{R_2}(R_2)=0\leq h \leq \g_{R_2}(R_2+\delta).$$

Since the function $\g_{R_2}$ is continuous and increasing, we guarantee the existence of a unique 
$\bar{R}_h \in [R_2,R_2+\delta]$ such that
$$h=\g_{R_2}(\bar{R}_h),$$
and consequently, the construction of the piece of $n$-dimensional catenoid in $\hiper^n\times \real$ which is vertical graph of the function
$$\g_{d_p, R_2},$$
over the set $\{q\in \hiper^n\times \{0\}; \; R_2\leq \rho_{d_p}(q)\leq \bar{R}_h\}$.\\

We note that this piece of catenoid has as boundary two hyperbolic $(n-1)$-spheres: $Z_{1,p}\subset \hiper^n\times \{0\}$ and $Z_{2,p}\subset \hiper^n\times \{h\}$, where $Z_{1,p}=Z_p$ and $Z_{2,p}$ has center $(d_p,h)$ and radius $\bar{R}_h$, see Figure \ref{fig:teoa-af2a-catenoide-def-ingles}.

\begin{figure}[htb]
\begin{center}
\includegraphics[scale=.6]{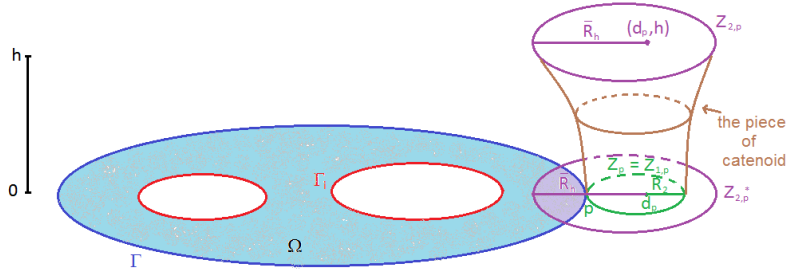}
\caption{The piece of catenoid (graph of $\g_{{d_p}, R_2}$) with boundary $Z_{1,p}\cup Z_{2,p}$ $(n=2)$.}
\label{fig:teoa-af2a-catenoide-def-ingles}
\end{center}
\end{figure}

Denoting by ${Z_{2,p}}^*$ the vertical projection of $Z_{2,p}$ in $\hiper^n\times \{0\}$, we see that ${Z_{2,p}}^*$ is a hyperbolic $(n-1)$-sphere of center $d_p=(d_p,0)$ and radius $\bar{R}_h\leq R_2+\delta$. This guarantees that the ball bounded by ${Z_{2,p}}^*$ is in the exterior of ${\displaystyle \bigcup_{i=1}^k \Gamma_i}$ in $\hiper^n\times \{0\}$, see again Figure \ref{fig:teoa-af2a-catenoide-def-ingles}.

The function $w_p$, defined in $\bar{\Omega}$ by
$$w_p(q)=\left\{
\begin{array}{rl}
\g_{d_p, R_2}(q), &  \mbox{ if } q\in\bar{\Omega} \mbox{ and it is in the region bounded by } {Z_{2,p}}^*\\
 & \mbox{ and }\Gamma,\\
h, & \mbox{ if } q\in\bar{\Omega} \mbox{ and it is not in the region bounded by } \\
 & \mbox{ } {Z_{2,p}}^* \mbox{ and }\Gamma, 
\end{array}
\right.$$
is a supersolution for the problem \textbf{(P$_0$)}, see again Figure \ref{fig:teoa-af2a-catenoide-def-ingles}. In fact, by construction, $w_p\in C^0(\bar{\Omega})$ and $w_p$ is the minimum of two supersolutions of \textbf{(P$_0$)}, namely, $w_p=\min \{\g_{{d_p}, R_2}, h\}$ in $\bar{\Omega}$. Moreover, $w_p(p)=0$.\\

\noindent \sub{Case 2}: $R_2=\infty$ and $n=2$.

In this case, we shall use the same notations adopted in Subsection \ref{subsec:exgrmin}-C.

We consider a horocycle $\Lambda_p$ passing through $p$ such that its mean convex side is contained in the closure of the exterior of $\Gamma$ in $\hiper^2\times\{0\}$. Up to a positive isometry of $\hiper^2\equiv\hiper^2\times\{0\}$, we may consider $p=(0,0,0)$ and $\Lambda_p$ as the horocycle parameterized by $Y(0,\cdot)$.

It follows from the hypotheses (\ref{eq:alfamin}) and (\ref{eq:alfar2}) that
\begin{equation}
0\leq h \leq \arcsec(e^{\delta})<\frac{\pi}{2}.
\label{eq:arcsec}
\end{equation}

We take
$$\ca_0^h=\left\{Y(u,\theta); \;\; (u,\theta)\in [0,h]\times \left(\frac{\pi}{2},\frac{5\pi}{2}\right)\right\}.$$

$\ca_0^h$ is the piece of the surface $\ca_0$ between the horocycles $\Lambda_{1,p}\subset \hiper^2\times \{0\}$ and $\Lambda_{2,p}\subset \hiper^2\times \{h\}$, where $\Lambda_{1,p}=\Lambda_p$ and $\Lambda_{2,p}$ is parameterized by $Y(h,\cdot)$, see Figure \ref{fig:teoa-caso2-suphoro-def-ingles}.

\begin{figure}[htb]
\begin{center}
\includegraphics[scale=.6]{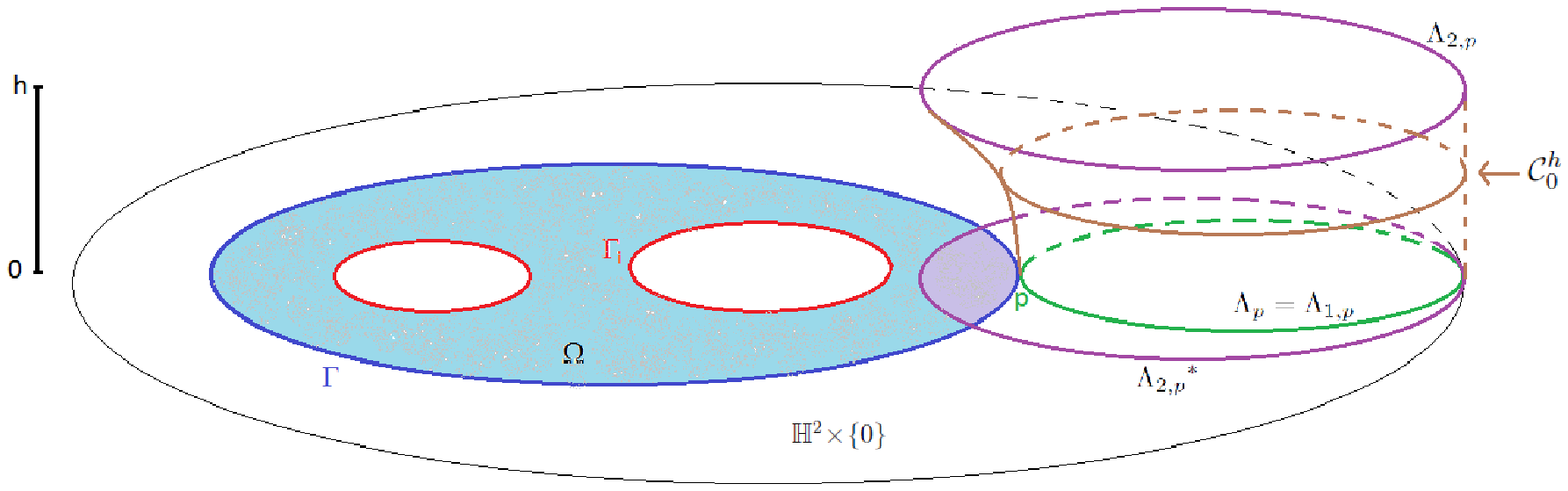}
\caption{$\ca_0^h$.}
\label{fig:teoa-caso2-suphoro-def-ingles}
\end{center}
\end{figure}

We denote by ${\Lambda_{2,p}}^*$ the vertical projection of $\Lambda_{2,p}$ in $\hiper^2\times \{0\}$. $\Lambda_p$ and ${\Lambda_{2,p}}^*$ are two horocycles of $\hiper^2\times \{0\}$ with the same asymptotic point $(0,1,0)\in\partial_{\infty}(\hiper^2\times \{0\})$, hence they are equidistant. Thus, the hyperbolic distance between the horocycles $\Lambda_p$ and ${\Lambda_{2,p}}^*$ is equal to the distance between the origin $(0,0,0)\in \Lambda_p$ and the point $q_h\in {\Lambda_{2,p}}^*$ nearest from the origin.

The point $q_h$ wanted must belong to the geodesic line of $\hiper^2\times\{0\}$ which passes through the origin and it is orthogonal to ${\Lambda_{2,p}}^*$. In that case, it is enough to find the intersection of the horocycle ${\Lambda_{2,p}}^*$ with the Euclidean line which passes through the origin, orthogonal to ${\Lambda_{2,p}}^*$.

Since $(0,0,0)=$ ${\displaystyle Y\left(0,\frac{3\pi}{2}\right)}$ and ${\displaystyle Y\left(h,\frac{3\pi}{2}\right)=}$ ${\displaystyle \left(0,\frac{\cos h - 1}{\cos h+1},h\right)}$, it follows that ${\displaystyle q_h=\left(0,\frac{\cos h-1}{\cos h+1},0\right)}$. Thus,

$$\begin{array}{cl}
\dist\left(\Lambda_p,{\Lambda_{2,p}}^*\right) & =\disto\left((0,0,0),q_h\right)\\
 & \\
 & {\displaystyle =\dist\left((0,0),\left(0,\frac{\cos h - 1}{\cos h+1}\right)\right)}\\
 & \\
 & =\ln (\sec h).
\end{array}$$

From (\ref{eq:arcsec}) follows that $0\leq \ln(\sec h) \leq \delta$, hence ${\displaystyle \dist\left(\Lambda_p,{\Lambda_{2,p}}^*\right) \leq \delta}$. This guarantees that ${\Lambda_{2,p}}^*$ is contained in the exterior of ${\displaystyle \bigcup_{i=1}^k \Gamma_i}$, see again Figure  \ref{fig:teoa-caso2-suphoro-def-ingles}.

We note that $\ca_0^h$ is the vertical graph of the function $\Upsilon$ over the closed region of $\hiper^2\times\{0\}$ between the horocycles $\Lambda_p$ and ${\Lambda_{2,p}}^*$.

The function $w_p$, defined in $\bar{\Omega}$ by
$$w_p(q)=\left\{
\begin{array}{rl}
\Upsilon(q), &  \mbox{ if } q\in\bar{\Omega} \mbox{ and it is in the region between } {\Lambda_{2,p}}^*\mbox{ and }\Gamma,\\
h, & \mbox{ if } q\in\bar{\Omega} \mbox{ and it is not in the region between } {\Lambda_{2,p}}^*\mbox{ and }\Gamma,  
\end{array}
\right.$$
is a supersolution for the problem \textbf{(P$_0$)}, see again Figure \ref{fig:teoa-caso2-suphoro-def-ingles}. In fact, by construction, $w_p\in C^0(\bar{\Omega})$ and $w_p$ is the minimum of two supersolutions of \textbf{(P$_0$)}, namely, $w_p=\min \{\Upsilon, h\}$ in $\bar{\Omega}$. Moreover, $w_p(p)=0$.

\vspace*{6 pt}

The uniqueness of the solution of \textbf{(P$_0$)} is guaranteed by the Maximum Principle. \hfill $\square$

\begin{rem}
We have used in the proof of Theorem \ref{teo:ac} a technique which is similar to the one used in Theorem 2.1 in \cite{er}. The main difference is with respect to the choose of the barriers for the case $n=2$ and $R_2=\infty$: while in \cite{er}, the authors have used affine planes of $\real^3$ (which are geodesic planes of $\real^3$) in the construction of the barriers for the case where the Jordan curve $\Gamma\subset\real^2$ satisfies the exterior circle condition of the radius $R_2=\infty$, we use, in the proof of Theorem \ref{teo:ac}, pieces of surfaces foliated by horocycles (which are not geodesic planes of $\hiper^2\times\real$) in the construction of the barriers for the corresponding case. See again Remark \ref{obs:ceeri} to understand the reason of that difference.
\label{obs:er}
\end{rem}

\section{Proof of Theorem \ref{teo:d}}
\label{sec:teod}

\indent\par We want to solve the following problem:
$$\mbox{\textbf{(P$_1$) }}\left\{
\begin{array}{l}
Q_0 \, u=0 \mbox{ in } \Omega, \;\; u\in C^2(\Omega)\cap C^0(\bar{\Omega}),\\
\left. u \right|_{\Gamma}=f, \;\; \left. u \right|_{\Gamma_i}=h, \;\; i=1,\; ...,\; k.  
\end{array}
\right.$$

We denote by $m:=\min f$ and by $M:=\max f$.

We note that the function $\phi\equiv h \mbox{ in } \bar{\Omega}$ is a supersolution for the problem \textbf{(P$_1$)}.

We define the set
$$S_1=\{s\in C^0(\bar{\Omega}); \; s \mbox{ is a subsolution of \textbf{(P$_1$)}, with } s\leq h\}.$$

The above set is not empty, because the function $s_1\equiv m \mbox{ in }\bar{\Omega}$ belongs to $S_1$.

We define the function
$$u(q):=\sup_{s\in S_1} s(q), \;\; q\in \bar{\Omega}.$$

It follows from Perron process that $u\in C^2(\Omega)$ and it satisfies $Q_0 \, u=0$. Moreover, by Perron process, we need to show that each point in $\partial\Omega$ admits barriers for the
problem \textbf{(P$_1$)}.

It follows from the convexity of $\Gamma$ and from the continuity of $f$  that we may use Example 4.1 of \cite{st:asymp}, in the case $n=2$, or Theorem 5.4 of \cite{st:hnxr}, in the case $n\geq 3$, for the construction of the barriers at each point $p\in \Gamma$, in the sense of Definition \ref{def:barreiras}-1, see again Remark \ref{obs:barreiras}. (Note that the condition ``$f$ is bounded'' is automatically satisfied, because $f$ is a continuous function defined over the compact set $\Gamma$.)

Now, for each point in ${\displaystyle \bigcup_{i=1}^k \Gamma_i}$, the barriers will be obtained in the sense of Definition \ref{def:barreiras}-2.\\

Given $i=1, \; ..., \;k$, let us consider a point $p\in \Gamma_i$.
	
We know that the function $w_p\equiv h$ in $\bar{\Omega}$ is a supersolution of \textbf{(P$_1$)} such that $w_p(p)=h$.\\

For the construction of a subsolution $v_p$ of \textbf{(P$_1$)} such that $v_p(p)=h$, we may assume, without loss of generality, that $m=0$, since the vertical translation is a isometry of $\hiper^n\times\real$. We may consider a hyperbolic $(n-1)$-sphere $S_p$ with radius $R$, passing through $(p,h)$ and contained in the region of the slice $\hiper^n\times \{h\}$ bounded by the hypersurface $\Gamma_i+h \, e_{n+1}$, where $e_{n+1}=(0,...,0,1)$. We denote by $c_p$ the hyperbolic center of $S_p$ and by ${c_p}^*$ the vertical projection of $c_p$ on the slice $\hiper^n\times \{0\}$.

It follows from the hypotheses (\ref{eq:oscn}) and (\ref{eq:alfarn}) that
$$\g_R(R)=0=m \leq M\leq h \leq \g_R(R+\delta).$$

Since the function $\g_{R}$ is continuous and increasing, we guarantee the existence of a unique $R_h \in [R,R+\delta]$ such that
$$h=\g_R(R_h).$$

Proceeding analogously to the initial part of the proof of Theorem \ref{teo:ac}, we obtain that the function $v_p$, defined by $v_p=\max \{-\g_{{c_p}^*, R}+ h, 0\}$ in $\bar{\Omega}$, is a subsolution for the problem \textbf{(P$_1$)} and $v_p(p)=h$.\\

The uniqueness of the solution of \textbf{(P$_1$)} is guaranteed by the Maximum Principle. \hfill $\square$

\begin{rem}
Remark \ref{obs:cceconvex} shows that there exists a $C^0$ hypersurface in $\hiper^n$ which is homeomorphic to $\esf^{n-1}$, satisfies the exterior sphere condition of some radius $0<\rho<\infty$ and it is not convex. Remark \ref{obs:checonvex} shows that there exists a $C^0$ hypersurface in $\hiper^n$ which is homeomorphic to $\esf^{n-1}$, satisfies the exterior sphere condition of radius $\rho=\infty$ and it is not convex. From this follows that Theorem \ref{teo:d} is not a generalization of Theorem \ref{teo:ac}. 
\label{obs:convexcee}
\end{rem}


\section{Proof of Proposition \ref{prop:inexn}}
\label{sec:propinexn}

\indent\par In this proof, we shall use a similar technique to the one observed in the final part of Theorem 5.1 in \cite{st:asymp} (p. 328-329).

Assume by contradiction that there exists a graph $M$ as mentioned in Proposition \ref{prop:inexn}.

Take
$$M^*:=-M+h e_{n+1},$$
where $e_{n+1}=(0,...,0,1)$ and $-M$ is the reflection of $M$ with respect to the slice $\Pi_1$. ($M^*$ is isometric to $M$.) Then $M^*$ is a connected compact minimal vertical graph of some function $u$ over $\Omega$ such that $\left. u \right|_{\beta^*}=0$ and $\left. u \right|_{\alpha}=h$. Notice that $M^*=\graph u$ is above the slice $\Pi_1$.

Now choose an $n$-dimensional catenoid $\ce_r:=\Sigma_r\cup -\Sigma_r$, where $\Sigma_r$ is the $n$-dimensional half-catenoid in $\hiper^n\times\real$ with rotational axis $\{0\}\times \real$ and generator curve given by $\g_r$ (see again Formula (\ref{eq:gr})), with $r$ big enough (that is, with large ``neck''), disjoint from $M^*$.

Let $\ce_r(\varepsilon)=\ce_r+\varepsilon e_{n+1}$ be the $\varepsilon$-vertical translation of $\ce_r$, with $\varepsilon>0$ small enough.

Now we shrink the catenoid $\ce_r(\varepsilon)$ in a family of catenoids with the same axis, shrinking its ``neck'', that is, making $r$ to approach zero.

Then we shall find a first point of contact of $M^*$ with one of these catenoids, we say, $\ce_{r_0}(\varepsilon)$.

We claim that the first point of contact does not occur on the boundary of $M^*$. In fact, notice that the boundary of $M^*$ is
$$\partial M^* = \beta^* \cup (\alpha+h e_{n+1}),$$
with $\beta^*\subset \Pi_1$ and $(\alpha+h e_{n+1})\subset \Pi_2$. Clearly $\ce_{r_0}(\varepsilon)$ does not touch $\beta^*$. Now for see that $\ce_{r_0}(\varepsilon)$ does not touch $(\alpha+h e_{n+1})$, we recall that the height of each $n$-dimensional half-catenoid $\Sigma_r$ increases from $0$ to $\pi/(2n-2)$, when $r$ increases from $0$ to $\infty$, see Subsection \ref{subsec:exgrmin}-A. From this follows that the height of each $n$-dimensional half-catenoid $\Sigma_r$ is smaller than $\pi/(2n-2)$ and it decreases to zero when $r$ approaches zero.
Therefore the height of the part of the catenoid $\ce_r(\varepsilon)$ above the slice $\Pi_1$ is smaller than $\pi/(2n-2)$ for all $r>0$ and it decreases when $r$ approaches zero (recall that $\varepsilon>0$ is taken small enough). Hence $\ce_{r_0}(\varepsilon)$ does not touch $(\alpha+h e_{n+1})$, because $h\geq \pi/(2n-2)$.

Thus the first point of contact of $M^*$ with the catenoid $\ce_{r_0}(\varepsilon)$ occurs on the interior of $M^*$. But this is a contradiction by the Maximum Principle, because the height
of $M^*$ is greater than or equal to $\pi/(2n-2)$ and the height of the part of the catenoid $\ce_{r_0}(\varepsilon)$ above the slice $\Pi_1$ is smaller than $\pi/(2n-2)$. This shows the proposition. \hfill $\square$


\end{document}